\documentclass[runningheads]{llncs}
\usepackage{amsmath,amssymb,amsfonts,float}
\usepackage{booktabs} 
\usepackage{caption} 
\usepackage{subcaption} 
\usepackage{graphicx}
\usepackage{pgfplots}
\usepackage[all]{nowidow}
\usepackage[utf8]{inputenc}

\usepackage[]{biblatex}

\definecolor{blue}{HTML}{1F77B4}
\definecolor{orange}{HTML}{FF7F0E}
\definecolor{green}{HTML}{2CA02C}

\pgfplotsset{compat=1.14}

\usepackage{tikz}
\usetikzlibrary{er,positioning,bayesnet}
\usetikzlibrary{positioning}
\tikzset{set/.style={draw,circle,inner sep=0pt,align=center}}

\newtheorem{Theorem}[theorem]{Theorem}
\newtheorem{Definition}[theorem]{Definition}
\newtheorem{Example}[theorem]{Example}
\newtheorem{Lemma}[theorem]{Lemma}
\newtheorem{Proposition}[theorem]{Proposition}

\newtheorem{Remark}[theorem]{Remark}

\begin{document}
\title{On $\delta\mathbb{P}$-approximation Spaces}
%
%
\author{Huda Mohsin\inst{1} \and
Faik Mayah\inst{2}}
%
%
\institute{Department of Mathematics, College of Education for Pure Sciences, Wasit University, Wasit, Iraq
\email{hudamohsin93@gmail.com}\\ \and
Department of Physics, College of Science, Wasit University, Wasit, Iraq\\
\email{fmayah@uowasit.edu.iq}}
\maketitle              
\begin{abstract}
In order to deal with imprecision, ambiguity, and uncertainty in data analysis, Pawlak introduced rough set theory in 1982. This paper aims to expand the scope of basic set theory developed by presenting the notions of $\delta\mathbb{P}$-upper and $\delta\mathbb{P}$-lower approximations, that are based on the notion of $\delta\mathbb{P}$-open sets, we additionally examine a few of their fundamental characteristics.

\end{abstract}

\section{Introduction}
The ``rough set theory" \cite{1}, a mathematical technique for coping with ambiguity or uncertainty, is attributed to Pawlak. Rough set theory and applications have substantially advanced since 1982. Rough set theory has several uses, particular, when analyzing data  both cognitive sciences and artificial intelligence \cite{2}, \cite{3}, \cite{4}. Pawlak and Skowron have recently published several fundamental ideas in rough set research as well as a number of applications \cite{5, 6}. An extension of set theory known as rough set theory \cite{5, 8} describes a portion of the universe as being described by two ordinary sets known as the lower and upper approximation. Initially, an equivalence was used to introduce Pawlak provides definitions for  both upper and lower approximations. Pawlak and Skowron \cite{5, 6} derived numerous intriguing characteristics of the lower and upper approximations in accordance with  the equivalence relations, the equivalence relation, on the other hand, seems to be a strict requirement that might restrict the suitability of the rough set model proposed by Pawlak. Equivalence relation or partition have undergone numerous extensions in recent years by being  replaced by concepts like binary relations \cite{10, 12}, neighborhood systems, by using a general relation, Abu-Donia \cite{17} talked about three distinct upper(lower) approximation types based on the appropriate neighborhood. These types were then used to create a collection of limited binary relationships in in two different methods. This theory, which primarily depends on a specific topological structure, has been extremely successful in  many fields of real-world applications. Numerous papers on generalizing and interpreting rough sets have been written \cite{13}, \cite{14}, \cite{15}, \cite{18}, \cite{20}, \cite{21}, \cite{23}, \cite{24}, \cite{25}. Weginger's generalization of rough sets is among the most important \cite{16} in 1989 when he introduced the idea of topological rough sets. The closure and interior operators of topological spaces were used to define the approximation in this generalization, which began with a topological space. The notion of $\delta\mathbb{P}$-open sets was first presented in \cite{11}. This work presents and investigates the concept of the $\delta\mathbb{P}$-approximation space. These areas aid in the development of a new categorization for the cosmos. Additionally, we explore the idea of "$\delta\mathbb{P}$-lower" furthermore "$\delta\mathbb{P}$-upper" approximation. Rough sets are compared to this idea as part of our study of $\delta\mathbb{P}$-rough sets. we also provide some opposition examples.

\section{Fundamental concepts $\delta\mathbb{P}$-open sets and topology}
  A pair$ (\mathbb{Y},\tau)$ that consists of a set $\mathbb{Y}$ and family $\tau$ of subset of $\mathbb{Y}$ that match the following criteria is topological space \cite{10}:
 \begin{enumerate}
 \item[A1.] $\emptyset, \mathbb{Y} \in  \tau$,
\item[A2.] $\tau$ is closed under  an arbitrary union,
\item[A3.] $\tau$ is closed  Under finite intersection .
 \end{enumerate}
The pair $(\mathbb{Y}, \tau)$ is referred to as a topological space, and the subset of $\mathbb{Y}$ that belongs to $\tau$ is referred to as open sets in the space. The opposite of the pair is called the closed subsets of $\mathbb{Y}$ are those that fall into the family $\tau$, the  family $\tau$ of open subsets of $\mathbb{Y}$ is often referred to as topology of $\mathbb{Y}$.\\ 
$\overline{S} $=$ \bigcap \{W \subseteq \mathbb{Y} : S \subseteq W$ and $W$ is closed $\}$ is known as $\tau$-closure of a subset $S\subset \mathbb{Y}$.\\
It appears that  the smallest closed subset of $\mathbb{Y}$ that contains $S$ is $\overline{S}$. Keep in mind that $S$ is only closed if and when  $S=\overline{S}$.\\ 
${S}^\circ = \bigcup \{V \subseteq  \mathbb{Y}:V \subset S $ and $V$ is open $\}$ is known $\tau$-interior of a subset $S \subseteq \mathbb{Y}$.\\
 It appears that ${S}^\circ$ is the union of all $\mathbb{Y}$ open subsets that contain in $S$. Keep in mind that $S$ is only open if and when $S =S^\circ$. Additionally, the $\tau $ 
boundary of a subset $S \subseteq \mathbb{Y}$ is denoted as $b(S)=\overline{S}-{S}^\circ$.\\
 Assume that $S$ is a subset of $\mathbb{Y}$, a topological space. Let ${S}^\circ$,$\overline{S}$, and $b(S)$ be interior, closure, and boundary of $S$, In that order . If $b(S)=\emptyset$ then $S$ is exact , if not, $S$ is rough. If and only if $\overline{S}={S}^\circ$, then it is evident that $S$ is exact.\\
 
\begin{Definition}
\cite{3} 
Let $(\mathbb{Y}, \tau)$ be a topological space the subset $ S \subseteq  \mathbb{Y}$ is referred to as Preopen if $S \subseteq Int(Cl(S))$.\\
 The opposite of Preopen set is Preclosed set. As we  that  indicate the set of all Preopen and Preclosed sets by $\mathbb{P}O(\mathbb{Y})$ and $\mathbb{P}C(\mathbb{Y})$.
 \end{Definition}

 \begin{Remark}
 Every topological space $(\mathbb{Y},\tau)$, has the property $\tau \subseteq \mathbb{P}O(\mathbb{Y})$.
 \end{Remark}

 \begin{Definition} 
 \cite{9}
Let $S$ be a subset of a topological space $ (\mathbb{Y},\tau)$. The $\delta$- closure of $S$ is defined as follow $cl_\delta (S) = \{x \in \mathbb{Y} : S \cap int(cl(A)) \neq \emptyset, A \in \tau$ and $ y \in A\}$. A set $S$ is referred to as $\delta$-closed if $S= cl_\delta(S)$.The opposite of a $\delta$-closed set is $\delta$-open.\\
Observe that $int_\delta (S) = \mathbb{Y} \backslash cl_\delta (\mathbb{Y} \backslash  S)$.
 \end{Definition}

\begin{Definition}
 \cite{11} 
 If $S$ is a subset of a topological space $(\mathbb{Y},\tau)$ and $S \subseteq int(cl_{\delta}(S))$, then $S$ is $\delta \mathbb{P}$-open.\\ 
$\delta\mathbb{P}O(\mathbb{Y})$ is the family of all $\delta\mathbb{P}$-open sets of $\mathbb{Y}$. $\delta\mathbb{P}$-closed is the complement of $\delta\mathbb{P}$-open.\\ 
The $ \delta\mathbb{P}$-closure of $S $ is $\mathbb{P}$cl$_\delta(S)$, which is the intersection of all $\delta\mathbb{P}$-closed sets that contain $S$.\\ 
The $\delta\mathbb{P}$-interior of $S$ is represented by $\mathbb{P}$int$_\delta(S)$ and is the union of all $\delta\mathbb{P}$-open sets that are contained in $S$. 
 
 \end{Definition}

 \begin{Lemma}
    \cite{17} The following hold for a subset $S$  of a topological space $(\mathbb{X},\tau)$.
    \begin{enumerate}
        \item $\mathbb{P}cl_\delta(S) = S \cup cl(int_\delta(S))$.
        \item $\mathbb{P}int_\delta(S) = S \cap int(cl_\delta(S))$.
        \item $\mathbb{P}cl_\delta(\mathbb{P}int_\delta(S)) = \mathbb{P}int_\delta(S) \cup cl(int_\delta(S))$.
        \item $\mathbb{P}int_\delta(\mathbb{P}cl_\delta(S)) = \mathbb{P}cl_\delta(S) \cap int(cl_\delta(S))$.
    \end{enumerate} 
 \end{Lemma}

\section{Rough set}

Rough set theory was inspired by the  necessity of express subsets of a universe in terms of equivalency classes of a partition of that universe. A topological space known as approximation space $\mathbb{K} = (\mathbb{X}, \tau)$ is characterized by the partition, where $ \mathbb{R}$ is an equivalency relation and $\mathbb{X}$ is a set known as the universe \cite{1}. The terms chunks, granules, and elementary sets are also used to describe the equivalency classes of $\mathbb{R}$. To denote that an equivalency class contains an $x \in \mathbb{X}$, we will use $\mathbb{R}_x \subseteq \mathbb{X}$. Two operators in the approximation space are considered.
\begin{eqnarray*}
   \underline{\mathbb{R}}(S) &=& \{ x \in \mathbb{X} : \mathbb{R}_x \subseteq S\},\\
     \overline{\mathbb{R}}(S) &=& \{ x \in \mathbb{X} : \mathbb{R}_x \cap S \neq \emptyset\}.
    \end{eqnarray*}

Referred to as, respectively, the lower  and upper approximations of $S$ $\subseteq$ $\mathbb{X}$. Furthermore, let $NEG_{\mathbb{R}}(S) = \mathbb{X} - \overline{\mathbb{R}}(S)$ represent the negative region of $S$, and $POS_{\mathbb{R}}(S) = \underline{\mathbb{R}}(S) $ represent the positive region of $S$, $BN_{\mathbb{R}}(S) = \overline{\mathbb{R}}(S) - \underline{\mathbb{R}}(S)$ represent the boundary region of$ \mathbb{X}$.\\ 
If we assume that $\mathbb{X}$ is a finite, nonempty universe and that $S \subseteq \mathbb{X}$, then the accuracy measure can also be used to quantify the degree of completeness as follows:\\

\[
\propto_{\mathbb{R}}(S) =\frac{\mid\underline{\mathbb{R}}(S)\mid} {\mid\overline{\mathbb{R}}(S)\mid}, \indent S \neq  \emptyset,
    \]

Where $|\cdot|$ the set's cardinality is represented. The degree of knowledge completeness is attempted to be expressed via accuracy metrics. $\propto_{\mathbb{R}}(S)$ is capable to depict the size of the data sets boundary region, but the knowledge's organizational structure is more difficult to depict. Rough set theory has the essential advantage of being able to handle categories that, given a knowledge basis, cannot be properly specified. The rough sets framework can be used to measure the properties of potential data sets, we may quantify imprecision and convey topological characterization of it,  with the help of the following .

\begin{enumerate}
    \item  If $\underline{\mathbb{R}}(S) \neq \emptyset $   and  
 $\overline{\mathbb{R}}(S) \neq \mathbb{X}$, then $S$ is roughly $\mathbb{R}$- definable, and indicated by $RD(\mathbb{X})$.
 
    \item If $\underline{\mathbb{R}}(S) = \emptyset$ and $\overline{\mathbb{R}}(S) \neq \mathbb{X}$, then $S$ is internally $\mathbb{R}$-undefinable, and indicated by $IUD(\mathbb{X})$,
    
    \item  If $\underline{\mathbb{R}}(S) \neq \emptyset$ and $\overline{\mathbb{R}}(S) = \mathbb{X}$, then $S$ is externally $\mathbb{R}$- undefinable, and indicated by $EUD(\mathbb{X})$,
    
    \item If $\underline{\mathbb{R}}(S) = \emptyset$ and $\overline{\mathbb{R}}(S) = \mathbb{X}$, then $S$ is totally $\mathbb{R}$- undefinable, and indicated by $TUD(\mathbb{X})$.
\end{enumerate}
We can characterize rough sets in terms of the boundary region's size and structure by utilizing $\propto_{\mathbb{R}}(S)$  and the previously given categories. Viewed as a particular instance of relative sets, rough sets are associated with the notion of Belnap's logic \cite{7}.

\begin{Remark}
We indicate the relationship between a class of Preopen sets $\mathbb{P}$O$(\mathbb{X})$ and a topology $\tau$ on $\mathbb{X}$ that was utilized to obtain a subbase using $ \mathbb{R}_\mathbb{P}$. Furthermore, we represent $\mathbb{P}$- approximation space by $(\mathbb{X},\mathbb{R}_\mathbb{P})$.
\end{Remark}

\begin{Definition}  
If $ (\mathbb{X},\mathbb{R}_\mathbb{P})$ be a $\mathbb{P}$-approximation space, then the $\mathbb{P}$-lower (resp., $\mathbb{P}$-upper) approximation of  each nonempty subset $S$ of $\mathbb{X}$ follows: as:
\begin{eqnarray*}
\underline{\mathbb{R}}_\mathbb{P}(S) &=& \bigcup \{ V\in \mathbb{P}O(\mathbb{X}) :  V\subseteq S\} \\
\overline{\mathbb{R}}_\mathbb{P}(S) &=& \bigcap \{ W \in \mathbb{P}C(\mathbb{X}) : W \supseteq S \}
\end{eqnarray*}

We  are able to obtain the $\mathbb{P}$-approximation operator as shown below.  .
\begin{enumerate}

    \item  From the provided relation $\mathbb{R}$, find the right neighborhood $x\mathbb{R}$, where$x\mathbb{R}$ = $\{$ y: x $\mathbb{R}$ y$ \}$.
    \item Taking right neighborhoods $x\mathbb{R}$ as a sub-base to obtain the topology $\tau$.
    \item Preopen set family obtained by using open sets in topology $\tau$ "from Definition 1."
    \item To obtain $\mathbb{P}$- approximation operators, Use the set of all Preopen sets ( see Definition 4).
\end{enumerate}

\end{Definition}

\begin{Proposition} 
For every $ S \subseteq \mathbb{X}$ in any $\mathbb{P}$- approximation space $(\mathbb{X},\mathbb{R}_\mathbb{P})$ the following are hold :
\begin{enumerate}
    \item $b(S)= \underline{Edg}(S) \cup \overline {Edg}(S)$.
    \item $\mathbb{P}b(S)= \mathbb{P}\underline{Edg}(S) \cup \mathbb{P} \overline{Edg}(S)$.
\end{enumerate}

\end{Proposition}

\begin{proof}
 (2) It the follows from \\
$\mathbb{P}b(S)$ = $\overline{\mathbb{R}}_\mathbb{P}(S) - \underline{\mathbb{R}}_\mathbb{P}(S)$
=$ (\overline{\mathbb{R}}_\mathbb{P}(S) - S)$  $\cup$  $( S - \underline{\mathbb{R}}_\mathbb{P}(S))$
=$\mathbb{P}\underline{Edg}(S)$  $\cup$  $\mathbb{P} \overline{Edg}$.
\end{proof}

\begin{Definition}
      Assume that $S$ $\subseteq \mathbb{Y}$  and that $(\mathbb{Y},\mathbb{R}_\mathbb{P})$ is a $\mathbb{P}$ - approximation space and. Then there are the memberships $\underline{\in}$, $\overline{\in}$, $\underline{\in}_\mathbb{P}$, and $\overline{\in}_\mathbb{P}$, which are defined as , strong, weak, $\mathbb{P}$-strong, and $\mathbb{P}$-weak memberships respectively . 
\begin{enumerate}
    \item $ y$ $\underline{\in}$ $S $  iff    $y$ $\in$ $\underline{\mathbb{R}}(S)$,
    \item $y$ $ \overline{\in}$$ S $  iff   $y$ $\in$ $\overline{\mathbb{R}}(S)$,
    \item $ y  \underline{\in} _\mathbb{P} S $ iff  $ y \in \underline{\mathbb{R}}_\mathbb{P}(S)$,
    \item $ y \overline{\in}_\mathbb{P} S $ iff    $y \in \overline{\mathbb{R}}_\mathbb{P}(S)$.
    
\end{enumerate}    
\end{Definition}

\begin{Remark}
 As stated by  Definition 5, $\mathbb{P}$- lower and $\mathbb{P}$- upper approximation of a set $ S \subseteq \mathbb{Y}$ is possible to write as
 \begin{eqnarray*}
 \underline{\mathbb{R}}_\mathbb{P}(S) &=& \{ y \in S : y \underline{\in}_\mathbb{P} S \},\\     
 \overline{\mathbb{R}}_\mathbb{P}(S) &=& \{ y \in S : y \overline{\in}_\mathbb{P} S \}.
 \end{eqnarray*}
\end{Remark}

\begin{Definition}
    Assume that $S \subseteq \mathbb{Y}$ and that $(\mathbb{Y},\mathbb{R}_\mathbb{p})$  is a $\mathbb{P}$ - approximation space. the $\mathbb{P}$ - accuracy measure of $S$, defined as follows;
\[
 \propto_{\mathbb{R}_\mathbb{P}(S)}  =  \frac{|\underline{\mathbb{R}}_\mathbb{P} (S)|}{|\overline{\mathbb{R}}_\mathbb{P} (S)|}, \indent \indent S \neq \emptyset.
 \]
\end{Definition}

\begin{Definition}

 The subset $N \subseteq \mathbb{Y}$ of any $\mathbb{P}$- approximation space $(\mathbb{Y},\mathbb{R}_\mathbb{P})$ is referred to as;  
    \begin{enumerate}
    \item
    If $\underline{\mathbb{R}}_\mathbb{P}(N) \neq \emptyset$ and $\overline{\mathbb{R}}_\mathbb{P}(N) \neq \mathbb{Y}$, then roughly $\mathbb{R}_\mathbb{P}$- definable,  and indicated by $\mathbb{P}RD(\mathbb{Y})$,
    
    \item 
     If $\underline{\mathbb{R}}_\mathbb{P}(N) = \emptyset$ and $\overline{\mathbb{R}}_\mathbb{P}(N) \neq \mathbb{Y}$, then internally $\mathbb{R}_\mathbb{P}$- undefinable,and indicated by $\mathbb{P}IUD(\mathbb{Y})$,
    
    \item
    if $\underline{\mathbb{R}}_\mathbb{P}(N) \neq \emptyset$ and $\overline{\mathbb{R}}_\mathbb{P}(N) = \mathbb{Y}$,then externally $\mathbb{R}_\mathbb{P}$- undefinable, and indicated by $\mathbb{P}EUD(\mathbb{Y})$,
    
    \item 
    If $\underline{\mathbb{R}}_\mathbb{P}(N) = \emptyset$ and $\overline{\mathbb{R}}_\mathbb{P}(N) = \mathbb{Y}$, then totally $\mathbb{R}_\mathbb{P}$- undefinable, and indicated by $\mathbb{P}TUD(\mathbb{Y})$.
    \end{enumerate}
    
\end{Definition}

\begin{Remark}
 For any $\mathbb{P}$ - approximation space $(\mathbb{Y},\mathbb{R}_\mathbb{P})$ the following hold:
 
 \begin{enumerate}

\item
$\mathbb{P}$RD$(\mathbb{Y})$ $\supseteq$ $ RD(\mathbb{Y})$,
\item 
$\mathbb{P}$IUD$(\mathbb{Y})$ $\subseteq$ $IUD(\mathbb{Y})$,
\item 
$\mathbb{P}$EUD$(\mathbb{Y})$ $\subseteq$ $EUD(\mathbb{Y})$,

\item 
$\mathbb{P}$TUD$(\mathbb{Y})$ $\subseteq$ $TUD(\mathbb{Y})$.

\end{enumerate}

\end{Remark}

\section{A novel approach to rough categorization using the $\delta\mathbb{P}$-open set}

\begin{Remark}
A subbase for a topology $\tau$ on  $\mathbb{X}$ and a class of $\delta \mathbb{P}O(\mathbb{X})$ of all $\delta \mathbb{P}$ -0pen sets by $\mathbb{R}_{\delta_\mathbb{P}}$ are indicated, along with the relationship that was utilized to obtain them. Furthermore, we designate the approximation space $\delta\mathbb{P} $   by $(\mathbb{X},\mathbb{R} _{{\delta \mathbb{P}}})$.

\end{Remark}

\begin{Example}
Assume a universe $\mathbb{X}$ = $\{u_1,u_2,u_3,u_4\}$  and a relation  $\mathbb{R}$ defined as $\mathbb{R} = \{(u_1,u_1),(u_1,u_2)\\,(u_1,u_3),(u_2,u_3),(u_3,u_4)\}$ thus $u_1\mathbb{R}= \{u_1,u_2,u_3\}$, $u_2\mathbb{R} = \{u_3\}$, $u_3\mathbb{R} = \{u_4\}$ as well $u_4\mathbb{R}= \emptyset$. Consequently, the topology related to this relationship is $\tau = \{\emptyset,\mathbb{X},\{u_3\},\{u_4\},\{u_3,u_4\},\{u_1,u_2,\\u_3\}\}$ as well $\delta\mathbb{P}O(\mathbb{X}) = \{\emptyset,\mathbb{X},\{u_1\},\{u_2\},\{u_3\},\{u_4\},\{u_1,u_2\},\{u_1,u_3\},\{u_1,u_4\},\{u_2,u_3\},\{u_2,\\u_4\},\{u_3,u_4\},\{u_1,u_2,u_3\},\{u_1,u_3,u_4\},\{u_1,u_2,u_4\},\{u_2,u_3,u_4\}\}$.is a $\delta\mathbb{P}$- approximation space.

\end{Example}

\begin{Definition}

Assume that ($\mathbb{X},\mathbb{R}_\delta$$_\mathbb{P}$) is $\delta\mathbb{P}$- lower approximation as well $\delta\mathbb{P}$- upper approximation for every nonempty subset $S$ of $\mathbb{X}$, the definition is:
 \begin{itemize}
 
\item 
$\underline{\mathbb{R}}_ {\delta}$$_ \mathbb{P}(S) = \bigcup  \{V : V \in \delta \mathbb{P}O(\mathbb{X}),  V \subseteq S\}$,
\item 
$\overline{\mathbb{R}}_ {\delta}$$ _\mathbb{P}(S) = \bigcap \{W : W \in \delta \mathbb{P}C(X),W \supseteq S\}$.

\end{itemize}

\end{Definition}

\begin{Definition}
    Assume that($\mathbb{X},\mathbb{R}_\delta$$_\mathbb{P}$) is $\delta\mathbb{P}$- accuracy measure of $S$  specified as follows
\[
  \propto_{\mathbb{R}_{\delta \mathbb{P}}}(S) = \frac{\mid \underline{\mathbb{R}}_ {\delta \mathbb{P}}(S)\mid}{\mid\overline{\mathbb{R}} _{\delta \mathbb{P}}(S)\mid},\indent  S \neq \emptyset.
  \]
  
\end{Definition}

\begin{Theorem}
Given any binary relation $\mathbb{R}$ on $\mathbb{X}$, which generates a topological space $(\mathbb{X}, \tau)$, we obtain $\underline{\mathbb{R}}(S) \subseteq \underline{\mathbb{R}}_\mathbb{P}(S) \subseteq \underline{\mathbb{R}}_\delta{}_\mathbb{P}(S) \subseteq S \subseteq \overline{\mathbb{R}}_ \delta{}_\mathbb{P}(S) \subseteq \overline{\mathbb{R}}_\mathbb{P}(S) \subseteq \overline{\mathbb{R}}(S)$. 
\end{Theorem}

\begin{proof}

$\underline{\mathbb{R}}(S) = \bigcup \{V \in \tau : V \subseteq S\} \subseteq \bigcup \{V \in \mathbb{P}O(X) :  V\subseteq S\} = \underline{\mathbb{R}}_\mathbb{P}(S) \subseteq \bigcup \{V \in \delta \mathbb{P} O(\mathbb{X}) : V \subseteq S\} =\underline{\mathbb{R}}_\delta{}_\mathbb{P}(S) \subseteq S$, that is,  $ \underline{\mathbb{R}}(S) \subseteq \underline{\mathbb{R}} _\mathbb{P}(S) \subseteq \underline{\mathbb{R}}_\delta{}_\mathbb{P}(S) \subseteq S$.\\
Furthermore, $\overline{\mathbb{R}}(S) = \bigcap \{W \in \tau ^ {c} :W \supseteq  S\} \supseteq \bigcap \{W \in \mathbb{P}C(\mathbb{X}) : W \supseteq S\} = \overline{\mathbb{R}}_\mathbb{P}(S) \supseteq \bigcap \{W \in \delta \mathbb{P}C(\mathbb{X}) : W \supseteq S\} = \overline{\mathbb{R}}_\delta{}_\mathbb{P}(S) \supseteq S$, that is, $\overline{\mathbb{R}}(S) \supseteq \overline{\mathbb{R}}_\mathbb{P}(S) \supseteq \overline{\mathbb{R}}_\delta{}_\mathbb{P}(S) \supseteq S$.\\
 Consequently, $ \underline{\mathbb{R}}(S) \subseteq \underline{\mathbb{R}}_\mathbb{P}(S) \subseteq \underline{\mathbb{R}}_\delta{}_\mathbb{P}(S) \subseteq  S  \subseteq \overline{\mathbb{R}}_\delta{}_\mathbb{P}(S) \subseteq \overline{\mathbb{R}}_\mathbb{P}(S) \subseteq \overline{\mathbb{R}}(S)$.
\end{proof}

\begin{Definition}

\label{def_01}

Assume that the  $\delta\mathbb{P}$- approximation space is ($\mathbb{X},\mathbb{R}_\delta$$_\mathbb{P}$). With consider to any $S \subseteq \mathbb{X}$, the universe $\mathbb{X}$ can be divided into 24 areas as follows.

\begin{enumerate}

\item  $\underline{Edg}(S) = S - \underline{\mathbb{R}}(S)$, which is the internal edg of $S$. 
\item 
 $\mathbb{P}\underline{Edg}(S) = S - \underline{\mathbb{R}}{}_\mathbb{P}(S)$, which is the $\mathbb{P}$-internal edg of $S$.
\item 
 $\delta \mathbb{P}\underline{Edg}(S) = S - \underline{\mathbb{R}}_\delta{}_\mathbb{P}(S)$, which is the $\delta\mathbb{P}$- internal edg of $S$ . 
\item 
 $\overline{Edg}(S) = \overline{\mathbb{R}}(S) - S$, which is the external edg of $S$.
\item 
$\mathbb{P}\overline{Edg}(s) =\overline{\mathbb{R}}_\mathbb{P}(S) - S$, which is the $\mathbb{P}$- external edg of $S$. 
\item 
 $\delta \mathbb{P}\overline{Edg}(S) = \overline{\mathbb{R}}_\delta{}_\mathbb{P}(S) - S$, which is the $\delta\mathbb{P}$- external edg of $S$. 
\item 
  b(S) = $\overline{\mathbb{R}}(S) - \underline{\mathbb{R}}(S)$, which is the  boundary of $S$. 

\item
 $\mathbb{P}b(S)$ = $\overline{\mathbb{R}}_\mathbb{P}(S) -\underline{\mathbb{R}}_\mathbb{P}(S)$, which is the $\mathbb{P}$- boundary of $S$.
\item
 $\delta \mathbb{P}b(S) = \overline{\mathbb{R}}_\delta{}_\mathbb{P}(S) - \underline{\mathbb{R}}_\delta{}_\mathbb{P}(S)$, which is the $\delta\mathbb{P}$- boundary of $S$. 
\item 
 $ext(S) = \mathbb{X} - \overline{\mathbb{R}}(S)$, which is the exterior of $S$.
\item 
 $\mathbb{P}ext(S) = \mathbb{X} - \overline{\mathbb{R}}_\mathbb{P}(S)$, which is the $\mathbb{P}$- exterior of $S$.
\item 
 $\delta \mathbb{P}ext(S) = \mathbb{X} - \overline{\mathbb{R}}_\delta{}_\mathbb{P}(S)$, which is the $\delta\mathbb{P}$- exterior of $S$. 
\item
$\overline{\mathbb{R}}(S) - \underline{\mathbb{R}}_\mathbb{P}(S)$.
\item 
$\overline{\mathbb{R}}(S) - \underline{\mathbb{R}}_\delta{}_\mathbb{P}(S)$.
\item 
$\overline{\mathbb{R}}(S) - \overline{\mathbb{R}}_\delta{}_\mathbb{P}(S)$.
\item 
$\overline{\mathbb{R}}_\mathbb{P}(S) - \underline{\mathbb{R}}(S)$.
\item 
$\overline{\mathbb{R}}_\mathbb{P}(S) - \underline{\mathbb{R}}_\delta{}_\mathbb{P}(S)$.
\item 
$\overline{\mathbb{R}}_\mathbb{P}(S) - \overline{\mathbb{R}}_\delta{}_\mathbb{P}(S)$.
\item 
$\underline{\mathbb{R}}_\mathbb{P}(S) - \underline{\mathbb{R}}(S)$.
\item 
$\overline{\mathbb{R}}_\delta{}_\mathbb{P}(S) - \underline{\mathbb{R}}_\mathbb{P}(S)$.
\item 
$\overline{\mathbb{R}}_\delta{}_\mathbb{P}(S) - \underline{\mathbb{R}}(S)$.
\item
$\underline{\mathbb{R}}_\delta{}_\mathbb{P}(S) - \underline{\mathbb{R}}_\mathbb{P}(S)$.
\item 
$\underline{\mathbb{R}}_\delta{}_\mathbb{P}(S) - \underline{\mathbb{R}}(S)$.
\item 
$\overline{\mathbb{R}}(S) - \overline{\mathbb{R}}_\mathbb{P}(S)$.
\end{enumerate}

\end{Definition}

\begin{Remark}
  An extension of the study of approximation space is the study of $\delta\mathbb{P}$ - approximation space (Figure 1). Due to the components of the areas [$ \underline{\mathbb{R}}_\mathbb{P}(S) - \underline{\mathbb{R}}(S)$], [$\underline{\mathbb{R}}_\delta{}_\mathbb{P}(S) - \underline{\mathbb{R}}_\mathbb{P}(S)$], and[$\underline{\mathbb{R}}_\delta{}_\mathbb{P}(S) - \underline{\mathbb{R}}(S)$] will be defined well in $S$, In Pawlak's approximation, however, this point was undefinable. Additionally, the component of the areas [$\overline{\mathbb{R}}(S) -\underline{\mathbb{R}}_\delta{}_\mathbb{P}(S)$],[$\overline{\mathbb{R}}_\mathbb{P}(S) - \overline{\mathbb{R}}_\delta{}_\mathbb{P}(S)$],additionally  [$\overline{\mathbb{R}}(S) - \overline{\mathbb{R}}_\mathbb{P}(S)$]  don't belong in $S$, even though Pawlak's approximation space doesn't provide these components a clear definition.\\
Figure \ref{fig_01} shows the above 24 areas. 
\end{Remark} 

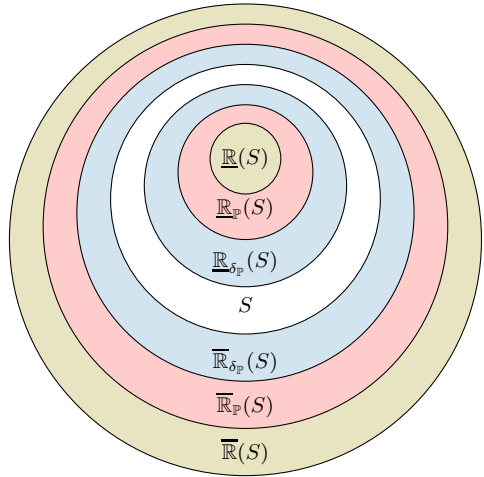
\begin{figure}[h]
\begin{center}
\begin{tikzpicture}[scale=0.9,transform shape]
\node[set,fill=olive!20,text width=7cm,label={[below=180pt of rea,text opacity=1]$\overline{\mathbb{R}}(S)$}] 
  (nat) at (0,-1.2)  (rea) {};
\node[set,fill=red!20,text width=6cm,label={[below=152pt of rea,text opacity=1]$\overline{\mathbb{R}}_{\mathbb{P}}(S)$}] 
  (nat) at (0,-1)  (rea) {};
\node[set,fill=blue!20,text width=5cm,label={[below=125pt of rea,text opacity=1]$\overline{\mathbb{R}}_{\delta_\mathbb{P}}(S)$}] 
  (nat) at (0,-0.8)  (rea) {};
\node[set,fill=white,text width=4cm,label={[below=95pt of rea,text opacity=1]$S$}] 
  (nat) at (0,-0.6)  (rea) {};
\node[set,fill=blue!20,text width=3cm,label={[below=67pt of rea,text opacity=1]$\underline{\mathbb{R}}_{\delta_\mathbb{P}}(S)$}] 
  (nat) at (0,-0.4)  (rea) {};
\node[set,fill=red!20,text width=2cm,label={[below=36pt of int]$\underline{\mathbb{R}}_{\mathbb{P}}(S)$}] 
  (int) at (0,-0.2)  {};
\node[set,fill=olive!20,text width=1cm] (nat) at (0,0) {$\underline{\mathbb{R}}(S)$};
\end{tikzpicture}
\end{center}
\caption{Showing the 24 areas given in Definition \ref{def_01}.}
\label{fig_01}
\end{figure}

\begin{Proposition}
If $S$ is any subset of $\mathbb{X}$, then the following holds for any  $\delta\mathbb{P}$- approximation space$(\mathbb{X},\mathbb{R}_\delta{}_\mathbb{P})$:
\begin{enumerate}

    \item 
b(S) = $\underline{Edg}(S) \cup \overline{Edg}(S)$.

\item
$\delta \mathbb{P}b(S) =\delta \mathbb{P} \underline{Edg}(S) \cup \delta \mathbb{P} \overline{Edg}(S)$.
\end{enumerate}

\begin{proof}
 (2) It the follows from \\
$ \delta \mathbb{P}b(S) = \overline{\mathbb{R}}_\delta{}_\mathbb{P}(S) - \underline{\mathbb{R}}_\delta{}_\mathbb{P}(S)  = (\overline{\mathbb{R}}_\delta{}_\mathbb{P}(S) - S) \cup (S - \underline{\mathbb{R}}_\delta{}_\mathbb{P}(S))
=\delta \mathbb{P}\underline{Edg}(S)  \cup \delta \mathbb{P} \overline{Edg}(S)$.
\end{proof}

\end{Proposition}

\begin{Definition}
Assume that $S \subseteq \mathbb{Y}$ and that ($\mathbb{X},\mathbb{R}_\delta{}_\mathbb{P}$) is a $\delta \mathbb{P}$- approximation space. Then there are the memberships $\underline\in_{\delta{}\mathbb{P}}$ ,$\overline\in_{\delta{}\mathbb{P}}$, which are defined as, $\delta \mathbb{P}$ - strong and $\delta \mathbb{P}$ - weak membership respectively 
\begin{enumerate}
    \item 
     $y$  $\underline\in_{\delta{}\mathbb{P}}  S$ iff $y$ $\in \underline{\mathbb{R}}_{\delta{}\mathbb{P}} (S)$,  
     \item 
     $y$ $\overline\in_{\delta{}\mathbb{P}} S$ iff $y$ $\in \overline{\mathbb{R}}_{\delta{}\mathbb{P}} (S)$.

\end{enumerate}

\end{Definition} 

\begin{Remark}
    As stated by definition 11, $\delta\mathbb{P}$- lower and $\delta\mathbb{P}$- upper approximations of a set $ S \subseteq \mathbb{Y}$ is possible to write as:
    \begin{enumerate}
        \item 
         $\underline{\mathbb{R}}_{\delta{}\mathbb{P}}$ (S) = $\{$y $\in$ S : $ y$ $\underline\in_{\delta{}\mathbb{P}}$ S $\}$,
         \item 
         $\overline{\mathbb{R}}_{\delta{}\mathbb{P}}$ (S) = $\{$y $\in$ S : $y$  $\overline\in_{\delta{}\mathbb{P}}$ S $\}$.
    \end{enumerate}
    
\end{Remark}

\begin{Remark}
\label{rem_011}
    Assume that ($\mathbb{X},\mathbb{R}_\delta{}_\mathbb{P}$) is a $\delta\mathbb{P}$ - approximation space, $S \subseteq \mathbb{Y}$ .Then
    \begin{enumerate}
        \item 
        $ y$  $\underline\in     $S$  \Rightarrow$  $y$  $\underline\in_\mathbb{P}  $S$ \Rightarrow$  $y$  $\underline\in _{\delta{}\mathbb{P}} S$, 
        \item
        $y$ $\overline\in_{\delta{}\mathbb{P} }S \Rightarrow$ $y$ $\overline\in_\mathbb{P} S \Rightarrow$ $y$ $\overline\in S$ .
    \end{enumerate}
    
\end{Remark} 
The converse of Remark \ref{rem_011} It might not always be the case, as demonstrated by the example below.

\begin{Example}
    In example 1. Let $N$ =$\{u_2,u_4\}$,we have $u_2$ $\underline\in _\delta{}_\mathbb{P}$ $N$ but $u_3$ $\underline\notin_\mathbb{P}$ $N$. Let $N$ = $\{u_1,u_3\}$, $u_1$  $\underline\in_\mathbb{P}$ $N$ but  $u_1$ $\underline\notin$ $N$. Let $ N$ =$\{u_1,u_4\}$Then we have $u_2$ $\overline\in$ $N$ but $u_2$ $\overline{\notin}_\mathbb{P}$ $N$. Let $N$ = $\{u_3\}$, $u_1$ $\overline\in _\mathbb{P}$ $N$ but $u_1$ $\overline{\notin}_\delta{}_\mathbb{P}$ $N$.
\end{Example}

\begin{Example}
  We can deduce from example 1 with the following table, which displays the degree of accuracy measure $\propto_{\mathbb{R}}$$(S)$,$\mathbb{P}$ -accuracy measure $\propto_{{\mathbb{R}}{}_\mathbb{P}}$$(S)$ additionally $\delta\mathbb{P}$- accuracy measure $\propto_{{\mathbb{R}}{}_{\delta{}\mathbb{P}}}$$(S)$ for some subset of $\mathbb{X}$.
\end{Example}

\begin{table}[H]
    \centering
    \begin{tabular}{|c|c|c|c|}
    \hline
   \qquad Power Set \qquad  &\qquad $\propto_{\mathbb{R}}$$(S)$ \qquad&\qquad $\propto_{{\mathbb{R}}{}_\mathbb{P}}$$(S)$ \qquad&\qquad $\propto_{{\mathbb{R}}{}_\delta{}_\mathbb{P}}$$(S)$\\ \hline \hline
$\{u_1 \}$&   0 &   0&1\\ \hline
        $\{u_2\}$& 0&0&1 \\ \hline
        $\{u_3\}$& $1/3$&$1/3$&1\\ \hline
        $\{u_4\}$& 1& 1& 1\\  \hline
        $\{u_1,u_2\}$&0&0&1\\  \hline
        $\{u_1,u_3\}$& $1/3$&$2/3$&1\\ \hline
        $\{u_1,u_4\}$&$1/3$& $1/2$&1\\ \hline
        $\{u_2,u_3\}$&$1/3$&$2/3$&1\\ \hline
        $\{u_2,u_4\}$& $1/3$&$1/2$&1\\ \hline
        $\{u_3,u_4\}$&$1/2$&$1/2$&1\\ \hline
        $\{u_1,u_2,u_3\}$&1&1&1\\ \hline
        $\{u_1,u_2,u_4\}$& $1/3$& $1/3$&1\\ \hline
        $\{u_1,u_3,u_4\}$& $1/2$&$3/4$&1\\ \hline
        $\{u_2,u_3,u_4\}$&$1/2$&$3/4$ &1\\ \hline
    \end{tabular}
    \caption{Showing the degree of accuracy measure $\propto_{\mathbb{R}}$$(S) $, $\propto_{{\mathbb{R}}{}_\mathbb{P}}$$(S)$ and $\propto_{{\mathbb{R}}{}_\delta{}_\mathbb{P}}$$(S)$.}
    \label{tab:my_label}
\end{table}

The set $S$ =$\{u_3,u_4\}$  has a degree of precision of 50$\%$ according to the accuracy measure, and 100$\%$ according to the $\delta\mathbb{P}$  - accuracy measure. Furthermore, the set $N$=$\{u_1,u_3,u_4\}$  according to $\mathbb{P}$ - accuracy measure equal to 75$\%$ and according to $\delta\mathbb{P}$ - accuracy measure equal to 100$\%$.Thus, $\delta\mathbb{P}$-accuracy measures are superior to accuracy and $\mathbb{P}$-accuracy metrics.\\ We study $\delta\mathbb{P}$- rough inclusion, using the rough inclusion method that Novotny and Pawlak developed in \cite{19,20}.

\begin{Definition}
     Assume that ($\mathbb{X},\mathbb{R}_\delta{}_\mathbb{P})$ is a $\delta\mathbb{P}$  approximation space  where $S$,$N$ $\subseteq \mathbb{X}$. Then we state:
     \begin{enumerate}
     
     \item
     $S$ is $\delta \mathbb{P}$- roughly bottom  included in $N$ if  $\underline{\mathbb{R}}_{{\delta}\mathbb{P}}(S) \subseteq \underline{\mathbb{R}}_{\delta{}\mathbb{P}}(N)$,
     
     \item 
     $S$ is $\delta \mathbb{P}$- roughly top included in $N$ if $\overline{\mathbb{R}}_{\delta\mathbb{P}}(S) \subseteq \overline{\mathbb{R}}_{\delta{}\mathbb{P}}(N)$,
     
     \item 
     $S$ is $\delta \mathbb{P}$- roughly included in $N$ if (1) and (2).
     \end{enumerate}
     
\end{Definition}

\begin{Example}
 As shown in Example 1, $\{u_2,u_4\}$  is $\delta\mathbb{P}$- roughly bottom included in $\{u_1,u_2,u_4\}$. \\Furthermore $\{u_2,u_4\}$ is $\delta\mathbb{P}$- roughly top included in $\{u_1,u_2,u_4\}$ .  Additionally $\{u_2,u_4\}$ is $\delta\mathbb{P}$- roughly included in $\{u_1,u_2,u_4\}.$
\end{Example}

\begin{Definition}
    Assume that ($\mathbb{X},\mathbb{R}_{\delta{}\mathbb{P}}$) is $\delta\mathbb{P}$- approximation space, a subset $S$ of $\mathbb{X}$ is referred to as
    \begin{enumerate}
    
    \item
    
    $\mathbb{R}_{\delta{}\mathbb{P}}$- definable ($\delta\mathbb{P}$- exact) when $\overline{\mathbb{R}}_{\delta{}\mathbb{P}}(S) = \underline{\mathbb{R}}_{\delta{}\mathbb{P}}(S)$,
    
    \item 
    $\delta{}\mathbb{P}$ - rough when $\overline{\mathbb{R}}_{\delta{}\mathbb{P}}(S) \neq \underline{\mathbb{R}}_{\delta{}\mathbb{P}}(S).$
    
\end{enumerate}

\end{Definition}

\begin{Example}
    For any $\delta\mathbb{P}$- approximation space ($\mathbb{X},\mathbb{R}_{\delta{}\mathbb{P}}$) as in Example 1.  We have the set $\{u_2,u_3,u_4\}$ is $\delta\mathbb{P}$- exact. 
\end{Example}

\begin{Definition}
    The subset $S \subseteq \mathbb{X}$ of any $\delta\mathbb{P}$- approximation space ($\mathbb{X},\mathbb{R}_{\delta{}\mathbb{P}})$ is referred to as;  
    \begin{enumerate}
    \item
    If $\underline{\mathbb{R}}_{\delta{}\mathbb{P}}(S) \neq \emptyset$ and $\overline{\mathbb{R}}_{\delta{}\mathbb{P}}(X) \neq \mathbb{X}$, then roughly $\mathbb{R}_{\delta{}\mathbb{P}}$- definable,  and indicated by  $\delta \mathbb{P}RD(\mathbb{X})$,
    
    \item 
     If $\underline{\mathbb{R}}_{\delta{}\mathbb{P}}(S) = \emptyset$ and $\overline{\mathbb{R}}_{\delta{}\mathbb{P}}(S) \neq \mathbb{X}$, then internally $\mathbb{R}_{\delta{}\mathbb{P}}$- undefinable, and indicated by  $\delta \mathbb{P}IUD(\mathbb{X})$,
    
    \item
    if $\underline{\mathbb{R}}_{\delta{}\mathbb{P}}(S) \neq \emptyset$ and $\overline{\mathbb{R}}_{\delta{}\mathbb{P}}(S) = \mathbb{X}$,then externally $\mathbb{R}_{\delta{}\mathbb{P}}$- undefinable, and indicated by $\delta \mathbb{P}EUD(\mathbb{X})$,
    
    \item 
    If $\underline{\mathbb{R}}_{\delta{}\mathbb{P}}(S) = \emptyset$ and $\overline{\mathbb{R}}_{\delta{}\mathbb{P}}(S) = \mathbb{X}$, then totally $\mathbb{R}_{\delta{}\mathbb{P}}$- undefinable, and indicated by $\delta \mathbb{P}TUD(\mathbb{X})$.
    \end{enumerate}
    
\end{Definition}

\begin{Remark}
  Assume that ($\mathbb{Y},\mathbb{R}_\delta{}_\mathbb{P}$) is a $\delta \mathbb{P}$- approximation space. These are on hold :
  \begin{enumerate}
    
  \item 
  $\delta \mathbb{P}RD(\mathbb{Y}) \supseteq \mathbb{P} RD(\mathbb{Y}) \supseteq RD(\mathbb{Y})$,
  
  \item 
  $\delta \mathbb{P}IUD(\mathbb{Y}) \subseteq \mathbb{P} IUD(\mathbb{Y}) \subseteq IUD(\mathbb{Y})$,
  \item 
  $\delta \mathbb{P}EUD(\mathbb{Y}) \subseteq \mathbb{P} EUD(\mathbb{Y}) \subseteq EUD(\mathbb{Y})$,
  \item
  $\delta \mathbb{P}TUD(\mathbb{Y}) \subseteq \mathbb{P} TUD(\mathbb{Y}) \subseteq TUD(\mathbb{Y})$.
  
  \end{enumerate}
  
\end{Remark}

\begin{Example}
    As shown in Example 1, the set $\{u_1,u_2\}$ $\in$ $\delta\mathbb{P}RD(\mathbb{X})$ but $\{u_1,u_2\}$ $\notin$ $\mathbb{p}RD(\mathbb{X})$.The set $\{u_2\}$ $\in$ $\mathbb{P} IUD(\mathbb{X})$ but $\{u_2\}$ $\notin $ $\delta \mathbb{P}IUD(\mathbb{X})$. The set $\{u_2\}$ $\in$ $\mathbb{P} IUD(\mathbb{X})$ and $\{u_2\}$ $\in$ $\delta{}\mathbb{P}IUD(\mathbb{X})$. the set $\{u_1,u_3,u_4\}$ $\in$ $\mathbb{P}EUD(\mathbb{X})$ but $\{u_1,u_3,u_4\}$ $\notin$ $\delta\mathbb{P}EUD(\mathbb{X})$.
\end{Example} 

\begin{Proposition}
    Assume that the  $\delta\mathbb{P}$ - approximation space is ($\mathbb{Y},\mathbb{R}_\delta{}_\mathbb{P}$). After that 
    \begin{enumerate}
    \item 
    Each $\mathbb{P}$ - exact set in $\mathbb{Y}$ is $\delta\mathbb{P}$ - exact
    
    \item 
    Each $\delta\mathbb{P}$ - rough set in $\mathbb{Y}$ is $\mathbb{P}$ - rough
    \end{enumerate}
    
    \begin{proof}
     Evident.\\
    As demonstrated by the example that follows, the converse of every part of proposition 3, might not always hold.
    \end{proof}
\end{Proposition}

\begin{Example}
    Consider ($\mathbb{X},\mathbb{R}_\delta{}_\mathbb{P}$) as an $\delta\mathbb{P}$ - approximation space  for in example 1. Consequently the subset $\{u_1,u_2\}$ is $\delta\mathbb{P}$- exact but not $\mathbb{P}$-exact, while $\{u_2,u_3,u_4\}$ is P-rough but not $\delta\mathbb{P}$-rough.
\end{Example}

\begin{Proposition}
 Assuming that $S$,$N$ $\subseteq \mathbb{X}$  and any $\delta\mathbb{P}$- approximation space ($\mathbb{X},\mathbb{R}_{\delta{}\mathbb{P}}$).Next
\begin{enumerate}
\item 
$\underline{\mathbb{\mathbb{R}}}_{\delta{}\mathbb{P}}(S) \subseteq S \subseteq \overline{\mathbb{R}}_{\delta\mathbb{P}}(S)$,
\item
$\underline{\mathbb{R}}_\delta{}_\mathbb{P}(\emptyset)=\overline{\mathbb{R}}_\delta{}_\mathbb{P}(\emptyset) = \emptyset, \underline{\mathbb{R}}_\delta{}_\mathbb{P}(\mathbb{X}) = \overline{\mathbb{R}}_\delta{}_\mathbb{P}(\mathbb{X}) = \mathbb{X}$,
\item
If $S \subseteq N $ then $\underline{\mathbb{R}}_\delta{}_\mathbb{P}(S) \subseteq \underline{\mathbb{R}}_\delta{}_\mathbb{P}(N)$ and $\overline{\mathbb{R}}_\delta{}_\mathbb{P}(S) \subseteq \overline{\mathbb{R}}_\delta{}_\mathbb{P}(N)$
\end{enumerate}

\begin{proof}
{\color{white}.}

\begin{enumerate}
    \item
    Assume $x \in \underline{\mathbb{R}}_\delta{}_\mathbb{P}(S)$, meaning that $x \in \bigcap \{V \in \delta \mathbb{P}O(\mathbb{X}),V \subseteq S\}$. And after that, there $V_0$ $\in \delta \mathbb{P}O(\mathbb{X})$ in such a manner that $x \in $ V$_0$ $\subseteq S$. So $x \in S$, therefore $\underline{\mathbb{R}}_\delta{}_\mathbb{P}(S) \subseteq S$, furthermore, assume $x \in \mathbb{X}$ additionally, by definition of $\overline{\mathbb{R}}_\delta{}_\mathbb{P}(S) = \bigcup\{W \in \delta \mathbb{P}C(\mathbb{X}), S \subseteq W\}$, then x $\in$ W for everyone $W \in$ $\delta \mathbb{P} C(\mathbb{X})$. Therefore $ S \subseteq \overline{\mathbb{R}}_\delta{}_\mathbb{P}(S)$.
    \item  
    Adheres directly.
    \item 
    Assume $ x \in \underline{\mathbb{R}}_\delta{}_\mathbb{P}(S)$, meaning that $x \in \bigcup \{V \in  \delta \mathbb{P}O(\mathbb{X}), V \subseteq S\}$ however $S \subseteq N$, so $V \subseteq N$ and $x \in V$, then $x \in \underline{\mathbb{R}}_\delta{}_\mathbb{P}(N)$. Additionally let $ x \notin \overline{\mathbb{R}}_\delta{}_\mathbb{P}(N) $ this implies that $x \notin \bigcap \{W \in \delta \mathbb{P}C(\mathbb{X}), N \subseteq W\}$ afterward, there are  $W \in \delta \mathbb{P}C(\mathbb{X}), N \subseteq W$ and $x \notin W$  meaning that, there is $W \in \delta \mathbb{P}C(\mathbb{X}), S \subseteq N \subseteq W$ and  $x \notin W$ which suggest $x \notin \bigcap\{W \in \delta \mathbb{P}C(\mathbb{X}), S \subseteq W\}$, so $x \notin \overline{\mathbb{R}}_\delta{}_\mathbb{P}(S)$. Consequently $\overline{\mathbb{R}}_\delta{}_\mathbb{P}(S) \subseteq \overline{\mathbb{R}}_\delta{}_\mathbb{P}(N)$.
\end{enumerate}
\end{proof}
\end{Proposition}

\begin{Proposition}
Assuming that $S$,$N$ $\subseteq \mathbb{X}$  and any $\delta\mathbb{P}$- approximation space ($\mathbb{X},\mathbb{R}_\delta{}_\mathbb{P}$).Next:

\begin{enumerate}
    \item $\underline{\mathbb{R}}_\delta{}_\mathbb{P}(\mathbb{X} \backslash S) = \mathbb{X} \backslash \overline{\mathbb{R}}_\delta{}_\mathbb{P}(S)$,

    \item 
    $\overline{\mathbb{R}}_\delta{}_\mathbb{P}(\mathbb{X} \backslash S) = \mathbb{X}  \backslash \underline{\mathbb{R}}_\delta{}_\mathbb{P}(S)$,

    \item $\underline{\mathbb{R}}_\delta{}_\mathbb{P}(\underline{\mathbb{R}}_\delta{}_\mathbb{P}(S)) = \underline{\mathbb{R}}_\delta{}_\mathbb{P}(S)$,

    \item $\overline{\mathbb{R}}_\delta{}_\mathbb{P}(\overline{\mathbb{R}}_\delta{}_\mathbb{P}(S)) = \overline{\mathbb{\mathbb{R}}}_\delta{}_\mathbb{P}(S)$,

    \item $\underline{\mathbb{R}}_\delta{}_\mathbb{P}(\underline{\mathbb{R}}_\delta{}_\mathbb{P}(S)) \subseteq \overline{\mathbb{R}}_\delta{}_\mathbb{P}(\underline{\mathbb{R}}_\delta{}_\mathbb{P}(S))$,

    \item $\underline{\mathbb{R}}_\delta{}_\mathbb{P}(\overline{\mathbb{R}}_\delta{}_\mathbb{P}(S)) \subseteq \overline{\mathbb{R}}_\delta{}_\mathbb{P}(\overline{\mathbb{R}}_\delta{}_\mathbb{P}(S)$.
\end{enumerate}
\end{Proposition}

\begin{proof}
{\color{white}.}

\begin{enumerate}
    \item Assume $x \in \underline{\mathbb{R}}_\delta{}_\mathbb{P}(\mathbb{X} \backslash  S)$  meaning that $ x \in \bigcup \{V \in \delta \mathbb{P}O(\mathbb{X}), V \subseteq \mathbb{X} \backslash S\}$. Consequently, there V$_0$ $\in \delta \mathbb{P}O(\mathbb{X})$ in such a manner that $ x \in$ V$_0$ $\subseteq \mathbb{X} \backslash S$. And after that, there  $V^c_0$ such that $S \subset$ $V ^ c _0$, $V^c_0$ $\in \delta \mathbb{P}C(\mathbb{X})$. So $x \notin \overline{\mathbb{R}}_\delta{}_\mathbb{P}(S)$. So $x \in \mathbb{X} \backslash \overline{\mathbb{R}}_\delta{}_\mathbb{P}(S)$. Consequently $\underline{\mathbb{R}}_\delta{}_\mathbb{P}(\mathbb{X} \backslash S) = \mathbb{X} \backslash \overline{\mathbb{R}}_\delta{}_\mathbb{P}(S)$.

    \item Comparable to (1).

    \item  By definition $\underline{\mathbb{R}}_\delta{}_\mathbb{P}(S) = \bigcup\{V \in \delta \mathbb{P}O(\mathbb{X}), V \subseteq S\}$,That suggests that $\underline{\mathbb{R}}_\delta{}_\mathbb{P}(\underline{\mathbb{R}}_\delta{}_\mathbb{P}(S))$ = $\bigcup \{ V \in \delta \mathbb{P}O(\mathbb{X}), V \subseteq \underline{\mathbb{R}}_\delta{}_\mathbb{P}(S) \subseteq S\}\}$ = $\bigcup\{V \in \delta \mathbb{P}O(\mathbb{X}),V \subseteq S\}$ = $\underline{\mathbb{R}}_\delta{}_\mathbb{P}(S)$.

    \item $\overline{\mathbb{R}}_{\delta \mathbb{P}}(\overline{\mathbb{R}}_{\delta\mathbb{P} }(S)) = \overline{\mathbb{R}}_{\delta \mathbb{P}}(\mathbb{X} \backslash \underline{\mathbb{R}}_{\delta \mathbb{P}}(\mathbb{X} \backslash S)) = \mathbb{X} \backslash \underline{\mathbb{R}}_{\delta \mathbb{P}}(\mathbb{X} \backslash \underline{\mathbb{R}}_{\delta \mathbb{P}}(\mathbb{X} \backslash S))$. From (1), (2) and (3), we get $\overline{\mathbb{R}}_{\delta \mathbb{P}}(\overline{\mathbb{R}}_{\delta \mathbb{P}}(S)) = \mathbb{X} \backslash \underline{\mathbb{R}}_{\delta \mathbf{P}}(\mathbb{X} \backslash S) = \mathbb{X} \backslash (\mathbb{X} \backslash \overline{\mathbb{R}}_{\delta \mathbb{P}}(S)) = \overline{\mathbb{R}}_{\delta \mathbb{P}}(S)$.

    \item  Since $\underline{\mathbb{R}}_\delta{}_\mathbb{P}(S) \subseteq \overline{\mathbb{R}}_\delta{}_\mathbb{P}(\underline{\mathbb{R}}_\delta{}_\mathbb{P}(S))$ and by (3) we have $\underline{\mathbb{R}}_\delta{}_\mathbb{P}(\underline{\mathbb{R}}_\delta{}_\mathbb{P}(S)) = \underline{\mathbb{R}}_\delta{}_\mathbb{P}(S)$, then $\underline{\mathbb{R}}_\delta{}_\mathbb{P}(\underline{\mathbb{R}}_\delta{}_\mathbb{P}(S)) \subseteq \overline{\mathbb{R}}_\delta{}_\mathbb{P}(\underline{\mathbb{R}}_\delta{}_\mathbb{P}(S))$.

    \item Since $\underline{\mathbb{R}}_\delta{}_\mathbb{P}(\overline{\mathbb{R}}_\delta{}_\mathbb{P}(S)) \subseteq \overline{\mathbb{\mathbb{R}}}_\delta{}_\mathbb{P}(S)$ and by (4), we have $\overline{\mathbb{R}}_\delta{}_\mathbb{P}(\overline{\mathbb{R}}_\delta{}_\mathbb{P}(S)) = \overline{\mathbb{R}}_\delta{}_\mathbb{P}(S)$,then $\underline{\mathbb{R}}_\delta{}_\mathbb{P}(\overline{\mathbb{R}}_\delta{}_\mathbb{P}(S)) \subseteq \overline{\mathbb{R}}_\delta{}_\mathbb{P}(\overline{\mathbb{R}}_\delta{}_\mathbb{P}(S))$.
\end{enumerate}
\end{proof}

\begin{Proposition}
 Assume that ($\mathbb{X},\mathbb{R}_\delta{}_\mathbb{P}$) is a $\delta\mathbb{P}$- approximation apace and $S$,$N \subseteq \mathbb{X}$. Then 

 \begin{enumerate}
     \item $\underline{\mathbb{R}}_\delta{}_\mathbb{P} (S \cup N) \supseteq \underline{\mathbb{R}}_\delta{}_\mathbb{P}(S) \cup \underline{\mathbb{R}}_\delta{}_\mathbb{P}(N)$,

     \item $\overline{\mathbb{R}}_\delta{}_\mathbb{P}(S \cup N) \supseteq \overline{\mathbb{R}}_\delta{}_\mathbb{P}(S) \cup \overline{\mathbb{R}}_\delta{}_\mathbb{P}(N)$,

     \item $\underline{\mathbb{R}}_\delta{}_\mathbb{P}(S \cap N) \subseteq \underline{\mathbb{R}}_\delta{}_\mathbb{P}(S) \cap \underline{\mathbb{R}}_\delta{}_\mathbb{P}(N)$,

     \item $\overline{\mathbb{R}}_\delta{}_\mathbb{P}(S \cap N) \subseteq \overline{\mathbb{R}}_\delta{}_\mathbb{P}(S) \cap \overline{\mathbb{R}}_\delta{}_\mathbb{P}(N)$.
 \end{enumerate}
\end{Proposition}

\begin{proof}
{\color{white}.}
\begin{enumerate}
    \item Given that we have $S \subseteq S \cup N$ and $N \subseteq S \cup N$. And after that $\underline{\mathbb{R}}_\delta{}_\mathbb{P}(S) \subseteq \underline{\mathbb{R}}_\delta{}_\mathbb{P}(S \cup N)$ and $\underline{\mathbb{R}}_\delta{}_\mathbb{P}(N) \subseteq \underline{\mathbb{R}}_\delta{}_\mathbb{P}(S \cup N)$ by (3) in the Proposition 4, then $\underline{\mathbb{R}}_\delta{}_\mathbb{P}(S \cup N) \supseteq \underline{\mathbb{R}}_\delta{}_\mathbb{P}(S) \cup \underline{\mathbb{R}}_\delta{}_\mathbb{P}(N)$.

    \item (2), (3) and (4) the same as (1).
\end{enumerate}
\end{proof}

\begin{Theorem}
Assuming that $S$,$N$ $\subseteq \mathbb{X}$  and any $\delta\mathbb{P}$- approximation space ($\mathbb{X},\mathbb{R}_\delta{}_\mathbb{P}$)
  If $S$ is $\mathbb{R}_\delta{}_\mathbb{P}$-definable.The next items are then held.

 \begin{enumerate}
     \item $\underline{\mathbb{R}}_\delta{}_\mathbb{P}(S \cup N) = \underline{\mathbb{R}}_\delta{}_\mathbb{P}(S) \cup \underline{\mathbb{R}}_\delta{}_\mathbb{P}(N)$.

     \item $\overline{\mathbb{R}}_\delta{}_\mathbb{P}( S \cap N) = \overline{\mathbb{R}}_\delta{}_\mathbb{P}(S) \cap \overline{\mathbb{R}}_\delta{}_\mathbb{P}(N)$.
 \end{enumerate}
 \end{Theorem}

 \begin{proof}
{\color{white}.}
\begin{enumerate}
\item 
 Evidently $\underline{\mathbb{R}}_\delta{}_\mathbb{P}(S) \cup \underline{\mathbb{R}}_\delta{}_\mathbb{P}(N) \subseteq \underline{\mathbb{R}}_\delta{}_\mathbb{P}(S \cup N)$. To include the opposite, assume $x \in \underline{\mathbb{R}}_\delta{}_\mathbb{P}(S \cup N)$, that implies $x \in \bigcup\{V $ is $ \delta \mathbb{P}O(\mathbb{X}), V \subseteq S \cup N\}$. And after that, there $V_0$ $\in \delta \mathbb{P}O(\mathbb{X})$ in such a manner that $x \in$ $V_0$ $\subset S \cup N$. We present three instances:

\begin{enumerate}
    \item  
         If $V_0$ $\subset$ $S$, $x$ $\in$ $V_0$ and $V_0$ is a $\delta\mathbb{P}$-open  the set, then $x \in \underline{\mathbb{R}}_\delta{}_\mathbb{P}(S)$.
         
    \item
         If $V_0$ $\cap$ $S$ = $\emptyset$, then $V_0$ $\subseteq$ $N$ and $x \in$ $V_0$, thus $x \in \underline{\mathbb{R}}_\delta{}_\mathbb{P}(N)$.

    \item
         If $V_0$ $\cap S \neq \emptyset$. Since $x \in$ $V_0$ and $V_0$ is an $\delta\mathbb{P}$-open the set, then $x \in$ $\delta\mathbb{P}cl(S)$, each $V_0$  in the previously  mentioned condition, therefore $x \in \overline{\mathbb{R}}_{\delta{}\mathbb{P}}(S)$, then $x \in \underline{\mathbb{R}}_{\delta{}\mathbb{P}}(S)$, since $S$ is ${\delta{}\mathbb{P}}$- definable. Thus, in three instances x $\in \underline{\mathbb{R}}_{\delta{}\mathbb{P}}(S) \cup \underline{\mathbb{R}}_{\delta{}\mathbb{P}}(N)$.
\end{enumerate}

\item  Evidently $\overline{\mathbb{R}}_\delta{}_\mathbb{P}(S \cap N) \subseteq \overline{\mathbb{R}}_\delta{}_\mathbb{P}(S) \cap \overline{\mathbb{R}}_\delta{}_\mathbb{P}(N)$. We demonstrate the opposite inclusion, assume $x \in \overline{\mathbb{R}}_\delta{}_\mathbb{P}(S) \cap \overline{\mathbb{R}}_\delta{}_\mathbb{P}(N)$, then $x \in \overline{\mathbb{R}}_\delta{}_\mathbb{P}(S)$ denotes $x \in \underline{\mathbb{R}}_\delta{}_\mathbb{P}(S)$ and $x \in V \subseteq \mathbb{X}$, in which $V$ is an $\delta\mathbb{P}$-open the set and $x \in \overline{\mathbb{R}}_{\delta{}\mathbb{P}}(N)$ suggests for every $V$ $\in$ $\delta\mathbb{P}O(\mathbb{X})$, $V$ $\cap N \neq \emptyset$. Consequently $ V$  $\cap (S \cap N) = (V \cap S) \cap N = V \cap Y \neq \emptyset$. Therefore $x \in \overline{\mathbb{R}}_\delta{}_\mathbb{P}(S \cap N)$.
\end{enumerate}
\end{proof}

\begin{Theorem}
    Assuming that $S$,$N$ $\subseteq \mathbb{X}$  and any $\delta\mathbb{P}$- approximation space ($\mathbb{X},\mathbb{R}_\delta{}_\mathbb{P}$). Afterwards, the following are held.
    \begin{enumerate}
        \item 
    $\overline{\mathbb{R}}_\delta{}_\mathbb{P}(cl(S) \cup N) = cl(S) \cup \overline{\mathbb{R}}_\delta{}_\mathbb{P}(N)$,
    \item 
    $\underline{\mathbb{R}}_\delta{}_\mathbb{P}(int(S) \cap N) = int(S) \cap \underline{\mathbb{R}}_\delta{}_\mathbb{P}(N)$.
    \end{enumerate}
\end{Theorem}

    \begin{proof}
{\color{white}.}
    \begin{enumerate}
        \item In accordance with propositions 4 (1) and  6 (2), we  $cl(S) \subset \overline{\mathbb{R}}_{\delta \mathbb{P}}(cl(S))$.Then cl(S) $\cup \overline{\mathbb{R}}_{\delta \mathbb{P}}(N) \subset \overline{\mathbb{R}}_{\delta \mathbb{P}}(cl(S)) \cup \overline{\mathbb{R}}_{\delta \mathbb{P}}(N) \subset \overline{\mathbb{R}}_{\delta \mathbb{P}}(cl(S) \cup N)$. However, since cl(S) $\cup M \subset cl(S) \cup \overline{\mathbb{R}}_{\delta \mathbb{P}}(N)$ and the union of an $\delta\mathbb{P} $-open set and a closed set is $\delta \mathbb{P}$-closed, and after that $\overline{\mathbb{R}}_{\delta \mathbb{P}}(cl(S) \cup N) \subset \overline{\mathbb{R}}_{\delta\mathbb{P} }(cl(S) \cup \overline{\mathbb{R}}_{\delta \mathbb{P}}(N)) = cl(S) \cup \overline{\mathbb{R}}_{\delta\mathbb{P} }(N)$. Consequently $\overline{\mathbb{R}}_{\delta \mathbb{P}}(cl(S) \cup N) = cl(S) \cup \overline{\mathbb{R}}_{\delta \mathbb{P}}(N)$.

        \item Given that  an open set's intersection  with int(S) and an $\delta\mathbb{P}$-open set $\underline{\mathbb{R}}_\delta{}_\mathbb{P}(N)$ is $\delta\mathbb{P}$-open, int(S) $\cap   \underline{\mathbb{R}}_\delta{}_\mathbb{P}(N) = \underline{\mathbb{R}}_\delta{}_\mathbb{P}(int(S) \cap \underline{\mathbb{R}}_\delta{}_\mathbb{P}(N)) \subset \underline{\mathbb{R}}_\delta{}_\mathbb{P}(int(S) \cap N)$. However, by applying proposition 6 (3), $\underline{\mathbb{R}}_\delta{}_\mathbb{P}(int(S) \cap N) \subset \underline{\mathbb{R}}_\delta{}_\mathbb{P}(int(S)) \cap \underline{\mathbb{R}}_\delta{}_\mathbb{P}(N) \subset int(S) \cap \underline{\mathbb{R}}_\delta{}_\mathbb{P}(N)$. Consequently $\underline{\mathbb{R}}_\delta{}_\mathbb{P}(int(S) \cap N) = int(S) \cap \underline{\mathbb{R}}_\delta{}_\mathbb{P}(N)$.
    \end{enumerate}
 \end{proof}

\begin{Lemma}
   For any $\delta\mathbb{P}$-approximation space ($\mathbb{X},\mathbb{R}_\delta{}_\mathbb{P}$) furthermore, for everyone $c$,$d \in \mathbb{X}$, the state of $c \in \overline{\mathbb{R}}_\delta{}_\mathbb{P}(\{d\}$) and $d \in \overline{\mathbb{R}}_\delta{}_\mathbb{P}(\{c\}$) infers $\overline{\mathbb{R}}_\delta{}_\mathbb{P}(\{c\}) = \overline{\mathbb{R}}_\delta{}_\mathbb{P}(\{d\})$.
\end{Lemma}

   \begin{proof}
   According to the definition of $\delta\mathbb{P}$-upper approximation  a set is a $\delta\mathbb{P}$-closure of this set, furthermore $\delta \mathbb{P}cl(\{d\})$ is a $\delta\mathbb{P}$-closed set containing $c$ (according to the condition) but $\delta \mathbb{P}cl(\{c\})$ is the tiniest $\delta\mathbb{P}$-closed set containing $c$, thus $\delta \mathbb{P}cl(\{c\}) \subseteq \delta \mathbb{P}cl(\{d\})$. Therefore $\overline{\mathbb{R}}_\delta{}_\mathbb{P}(\{c\}) \subseteq \overline{\mathbb{R}}_\delta{}_\mathbb{P}(\{d\})$ By symmetry, the opposite inclusion occurs $\delta \mathbb{P}cl(\{d\}) \subseteq \delta \mathbb{P}cl(\{c\})$. therefore $\overline{\mathbb{R}}_\delta{}_\mathbb{P}(\{d\}) \subseteq \overline{\mathbb{R}}_\delta{}_\mathbb{P}(\{c\})$ we obtain $\overline{\mathbb{R}}_\delta{}_\mathbb{P}(\{c\}) = \overline{\mathbb{R}}_\delta{}_\mathbb{P}(\{d\})$.
\end{proof}

\begin{Lemma}
    Assume that ($\mathbb{X},\mathbb{R}_\delta{}_\mathbb{P}$) be a $\delta\mathbb{P}$-approximation space, where each $\delta\mathbb{P}$-open subset $S$ of $\mathbb{X}$ is $\delta\mathbb{P}$-closed, Then  $d \in \overline{\mathbb{R}}_\delta{}_\mathbb{P}(\{c\})$ therefore $c \in \overline{\mathbb{R}}_\delta{}_\mathbb{P}(\{d\})$ for every $c$,$d \in \mathbb{X}$.
\end{Lemma}

\begin{proof}
    If $c \notin \overline{\mathbb{R}}_\delta{}_\mathbb{P}(\{d\})$, then  there is a $\delta\mathbb{P}$-open set $V$ include $c$ such that $V$ $\cap \{d\} = \emptyset$ which suggests that $\{d\} \subseteq (\mathbf{X} \backslash V)$ but ($\mathbb{X} \backslash V$) is a $\delta\mathbb{P}$-closed set additionally is a $\delta\mathbb{P}$-open set does not include $c$, thus ($\mathbb{X} \backslash V$) $\cap$ $\{$c$\}$ = $\emptyset$. Therefore $d \notin \overline{\mathbb{R}}_\delta{}_\mathbb{P}(\{c\})$.
\end{proof}

\begin{Proposition}    
  Assume that ($\mathbb{X},\mathbb{R}_{\delta{}\mathbb{P}}$) be a $\delta\mathbb{P}$- approximation space, and all of them $\delta\mathbb{P}$-open subset $S$ of $\mathbb{X}$ is $\delta\mathbb{P}$-closed. After that, the family of sets $\{\overline{\mathbb{R}}_\delta{}_\mathbb{P}(\{c\})$ : $c \in S\}$ is a division of the set $\mathbb{X}$.
\end{Proposition}

  \begin{proof}
  If $c$,$d$,$f$ $\in S$ furthermore $f \in \overline{\mathbb{R}}_\delta{}_\mathbb{P}(\{c\}) \cap \overline{\mathbb{R}}_\delta{}_\mathbb{P}(\{d\})$, then $f \in \overline{\mathbb{R}}_\delta{}_\mathbb{P}(\{c\})$ furthermore $f \in \overline{\mathbb{R}}_\delta{}_\mathbb{P}(\{d\})$. Consequently, by Lemma 3, $c \in \overline{\mathbb{R}}_\delta{}_\mathbb{P}(\{f\})$ and $d \in \overline{\mathbb{R}}_\delta{}_\mathbb{P}(\{f\})$ furthermore by Lemma 4, as we have $\overline{\mathbb{R}}_\delta{}_\mathbb{P}(\{c\}) = \overline{\mathbb{R}}_\delta{}_\mathbb{P}(\{f\})$ and $\overline{\mathbb{R}}_\delta{}_\mathbb{P}(\{d\}) = \overline{\mathbb{R}}_\delta{}_\mathbb{P}(\{f\})$. Consequently $\overline{\mathbb{R}}_\delta{}_\mathbb{P}(\{c\}) = \overline{\mathbb{R}}_\delta{}_\mathbb{P}(\{d\}) = \overline{\mathbb{R}}_\delta{}_\mathbb{P}(\{f\})$. Therefore either $\overline{\mathbb{R}}_\delta{}_\mathbb{P}(\{c\}) = \overline{\mathbb{R}}_\delta{}_\mathbb{P}(\{d\})$ or $\overline{\mathbb{R}}_\delta{}_\mathbb{P}(\{c\}) \cap \overline{\mathbb{R}}_\delta{}_\mathbb{P}(\{d\}) = \emptyset$
\end{proof}


\section{Conclusions}
This paper introduces the $\delta\mathbb{P}$- approximation operator, a new class of approximations that we introduced using the $\delta\mathbb{P}$-open sets class. Furthermore, the $\delta\mathbb{P}$- approximation yields 24 unique granules of the discourse universe.The most extensive granulation based on closure and interior operator in topological spaces is used in our method, which is the class of $\delta\mathbb{P}$- open sets. Because of this, the accuracy measurements are higher than when using any kind of near-open sets, like $\alpha$-open, etc. There are some generalizations of significant characteristics of the traditional Pawlak's rough sets. Additionally, we used our approach to define the notion of rough membership function. It is an extension of the traditional rough membership function of Pawlak rough sets. In a decision information system, depending on a conditional attribute. The decision that has to be made can be done using the generalized rough membership function. Intelligent computational versions of granular beneficial computing is generated by the rough set approach to approximation of sets.

%
%
%

%
\end{document}